\documentclass[12pt]{amsart}
\usepackage{amsmath, amsthm, amscd, amsfonts}
\usepackage{graphicx}

\newfont{\aj}{eufm10 at12pt}
\newfont{\ajk}{eufm10 at10pt}
\theoremstyle{plain}
\newtheorem{theorem}{Theorem}[section]
\newtheorem{lemma}[theorem]{Lemma}

\newtheorem{corollary}[theorem]{Corollary}
\numberwithin{equation}{section}
\theoremstyle{definition}
\newtheorem{definition}[theorem]{Definition}

\newtheorem{remark}[theorem]{Remark}

\begin{document}

\title[Adding Machine Maps and Minimal Sets]{Adding Machine Maps and Minimal Sets for Iterated Function Systems}
\author[Mehdi Fatehi Nia]{ {{\bf {Mehdi Fatehi Nia}}
\\{\tiny{Department of Mathematics, Yazd University, 89195-741 Yazd, Iran}
\\e-mail: fatehiniam@yazd.ac.ir }}}

\maketitle

\begin{abstract}%
In this paper, we focus attention on extending the topological conjugacy of adding machine maps and minimal systems to iterated function systems. We provide necessary and sufficient conditions for an iterated function system  to be conjugated to
an adding machine map. It is proved that every minimal iterated function system which has some  non-periodic regular point is semi-conjugate to an adding machine map. Furthermore, we investigate the topological conjugacy of an infinite family of tent maps, as well as  the restriction of a map to  its $\omega-$limit set with an iterated  function system.
\end{abstract}

\section{Introduction}
Minimality is an important concept in the study
of dynamical systems. It describes the
irreducibility of a system from the topological point of view, which means that a minimal system has no proper subsystem \cite{[BIR]}. Adding machines
are a fundamental component of discrete dynamical systems \cite{[AB],[BK],[DHSA],[DHS]}.
Let $f :X\rightarrow X$ be a continuous map of a compact metric
space $X$. In \cite{[BK]} the authors give sufficient and necessary conditions for $f$ to be topologically conjugate to an adding machine map. \\Let us recall that an  Iterated Function System(\textbf{IFS}) $\mathcal{F}=\{X; f_{\lambda}|\lambda\in\Lambda\}$ is any family of continuous mappings $f_{\lambda}:X\rightarrow X,~\lambda\in \Lambda$, where $\Lambda$ is a finite nonempty set (see\cite{[B]}). \\Iterated function systems are used for the construction of deterministic fractals and have found numerous applications, in
particular to image compression and image processing \cite{[B]}. Important notions in
dynamics like attractors, minimality, transitivity, and shadowing have been extended to
\textbf{IFS} (see \cite{[BV],[GG],[GG1],[FN]}).
 \\The present paper concerns the minimality and adding machine maps  for  \textbf{IFS}.  In the next section, we introduce certain notation and we give some definitions, which will be used throughout the paper.  In Section $3$ we proceed
with the study of minimality and minimal sets for \textbf{IFS}. Theorem \ref{tha}, which  is the main result of this section,  shows that for a sequence $\alpha=(j_{1},j_{2},...)$ of integers greater than $1$, an \textbf{IFS} is topologically conjugate to adding machine map $g_{\alpha}$, if and only if conditions $(1)-(3)$ of Theorem \ref{tha} are satisfied.
In Section $4$ several lemmas lead to Theorem \ref{ts} which shows that:  If the \textbf{IFS}, $\mathcal{F}=\{X; f_{\lambda}|\lambda\in\Lambda\} $ is minimal, $X$ is infinite and $\mathcal{F}$ has some non-periodic regularly recurrent points, then there is a sequence $\alpha$ of prime numbers and a continuous surjective map $\pi:X\rightarrow  \Delta_{\alpha}$ such that $g_{\alpha}\circ\pi=\pi \circ f_{\lambda}$, for all $\lambda\in \Lambda$. Also, we will present several corollaries of this theorem and its proof.
\section{Preliminaries}
Let $(X,d)$ be a compact metric space and $\mathcal{F}=\{X; f_{\lambda}|\lambda\in\Lambda\}$ be an \textbf{IFS}. \\ We let    $\Lambda^{\mathbb{Z}_{+}}$ denote the set
of all infinite sequences $\{\lambda_{i}\}_{i \geq 1}$ of symbols belonging to $\Lambda$. For a given $\sigma= \{\lambda_{n}\}\in \Lambda^{\mathbb{Z}_{+}}$
we put $ \mathcal{F}_{\sigma}=\{\mathcal{F}_{\sigma_{n}}\}$ where \\$\mathcal{F}_{\sigma_{n}}=f_{\lambda_{n}} \circ \cdots \circ f_{\lambda_{2}}\circ  f_{\lambda_{1}},~~{n\geq 1}$.
 \\A sequence $\{x_{n}\}_{n\geq 1}$ in $X$ is
 called an orbit of the \textbf{IFS} $\mathcal{F}$, if there exists $\sigma\in \Lambda^{\mathbb{Z}_{+}}$ such that $x_{n+1}=f_{\lambda_{n}}(x_{n})$, for each $\lambda_{n}\in \sigma$.\\
For a subset $A\subset X$, we let $\mathcal{F}(A)=\bigcup_{\lambda\in\Lambda}f_{\lambda}(A)$.
\begin{definition}
Suppose that $(X,d)$ and $(Y,d^{'})$ are compact metric spaces. Let $ \mathcal{F}=\{X; f_{\lambda}|\lambda\in\Lambda\}$ and $ \mathcal{G}=\{Y; g_{\lambda}|\lambda\in\Lambda\}$ be two \textbf{IFS} such that  $f_{\lambda}:X\rightarrow X$ and  $g_{\lambda}:Y\rightarrow Y$ are continuous maps, for all $\lambda\in\Lambda$. Consider $\sigma\in\Lambda^{\mathbb{Z}_{+}}$, we say that $ \mathcal{F}_{\sigma}$ is topologically conjugate to $ \mathcal{G}_{\sigma}$  if there is a homeomorphism  $h:X\rightarrow Y$  such that $\mathcal{G}_{\sigma_{n}}\circ~h=h~ \circ \mathcal{F}_{\sigma_{n}} $, for all $n\geq 1$.\end{definition}
By the above definition, for $\sigma\in\Lambda^{\mathbb{Z}_{+}}$ we say that $ \mathcal{F}_{\sigma}$ is topologically conjugate to the map $g:Y\rightarrow Y$, if there is a homeomorphism  $h:X\rightarrow Y$  such that $g^{n}\circ~h=h~ \circ  \mathcal{F}_{\sigma_{n}} $, for all $n\geq 1$.
\\
Let $ \mathcal{F}=\{X; f_{\lambda}|\lambda\in\Lambda\}$ be an \textbf{IFS} and  $k$ be a positive integer.
Set \\$\mathcal{F}^{k}=\{X;f_{\lambda_{k}}\circ \cdots\circ f_{\lambda_{1}}|\lambda_{1},
...,\lambda_{k}\in\Lambda\}$. \\
It should be noted that for every $k>1$, $~\mathcal{F}^{k}$ is an \textbf{IFS} such that its functions and their number  are different from those in $\mathcal{F}$.
\begin{definition}\cite{[GHS]} Let $n$ be a positive integer and $\mathcal{F}=\{X; f_{\lambda}|\lambda\in\Lambda\} $. A nonempty closed set $M\subseteq X$ is called a $\mathcal{F}^{n}-$minimal set if
$\mathcal{F}_{\sigma_{n}}(M)=M$ and  $\mathcal{F}_{\sigma_{n}}(A)\neq A$, for all nonempty sets $A\subset M$ and all $\sigma\in \Lambda^{\mathbb{Z}_{+}}$.
The \textbf{IFS} $\mathcal{F}$ is minimal if the only minimal set for $~\mathcal{F}^{1}$ is the whole space $X$. \\
Equivalently, for a minimal iterated function system $\mathcal{F}=\{X; f_{\lambda}|\lambda\in\Lambda\} $, for any $x \in X$ the collection
of iterates $f_{\lambda_{k}} \circ \cdots \circ f_{\lambda_{1}} (x),~$ $k>0$ and $\lambda_{1},\ldots,\lambda_{k}\in\Lambda$, is dense in $X$.\end{definition}

\begin{definition}\label{df}(Adding machine map. \cite{[BK]})
Let $\alpha=(j_{1},j_{2},\ldots)$ be a sequence of integers where each $j_{i}\geq 2$. Let $\Delta_{\alpha}$ denote all sequences $(r_{1},r_{2},\ldots)$ where $r_{i}\in\{0,1,\cdots,j_{i}-1\}$ for each $i$. We put a metric $d_{\alpha}$ on $\Delta_{\alpha}$ given by $d_{\alpha}((r_{1},r_{2},\ldots),(s_{1},s_{2},\cdots))=\Sigma_{i=1}^{\infty}\frac{\delta(r_{i},s_{i})}{2^{i}},$ where $\delta(r_{i},s_{i})=1$ if $r_{i}\neq s_{i}$ and $\delta(r_{i},s_{i})=0$ if $r_{i}= s_{i}.$ Addition in $\Delta_{\alpha}$ is defined as follows:\\
$$(r_{1},r_{2},...)+(s_{1},s_{2},...)=(z_{1},z_{2},\ldots),$$ where $z_{1}=(r_{1}+s_{1})~ mod ~j_{1}$ and $z_{1}=(r_{2}+s_{2}+t_{1})~ mod ~j_{2}$. Here $t_{1}=0$ if $r_{1}+s_{1}<j_{1}$ and $t_{1}=1$ if $r_{1}+s_{1}\geq j_{1}$. Continue adding and carrying in this way for the whole sequence.\\ We define, the adding machine map, $g_{\alpha}:\Delta_{\alpha}\rightarrow \Delta_{\alpha}$ by $$g_{\alpha}((r_{1},r_{2},\ldots))=(r_{1},r_{2},...)+(1,0,0,\ldots).$$
If each $j_{i}=2$ then the system is called the dyadic adding machine.\end{definition}
Let $f:X\rightarrow X$ be a continuous map. A point $x$ is non-wandering point of $f$ if for any neighborhood $U$ of $x$ there exists
$n \in N$ such that $f^{n}(U) \cap U\neq \emptyset$. The set of all non-wandering points of $f$ is denoted as $\Omega(f)$, if $\Omega(f)=X$ then the systems is said to be non-wandering.\\
A sequence $\{x_{n}\}_{n\geq 0}$ is called a $\delta-$pseudo-orbit $(\delta\geq 0)$ of $f$ if
$d(f(x_{n}), x_{n+1})\leq \delta$ for all $n\geq 0$. A map $f$ is said to have the shadowing property if for any $\epsilon>0$ there is $\delta>0$ such that for every $\delta-$pseudo-orbit $\{x_{n}\}_{n\geq 0}$ of $f$, there exists $y\in X$ satisfying $d(f^{n}(y),x_{n})<\epsilon$ for all $n\geq 0$.\\
A dynamical system $(X,f)$ is sensitive if there is $\delta>0$ with the property that for every nonempty open set $U\subset X$, there is $n>0$ such that $diam(f^{n}(U))>\delta$\cite{[SO]}.
\section{Topological conjugacy of adding machine maps and minimal IFS}
 In this section we present some lemmas and
notations needed in the proof of our main results.
\begin{lemma}\label{leaa}
Let $n\geq 1$ and $M$ be a $\mathcal{F}^{n}-$minimal set, then \\$\mathcal{F}_{\sigma_{n}}(\mathcal{F}^{i}(M))=\mathcal{F}^{i}(M)$, for every positive integer $i$ and every $\sigma\in\Lambda^{\mathbb{Z}_{+}}$. Also,  there is $t\leq n$ such that $\mathcal{F}^{i}(M)$ is an $\mathcal{F}^{n}-$minimal set, for all $i\leq t-1$.
\end{lemma}
\begin{proof}
By the definition of $\mathcal{F}^{n}-$minimal set, it is enough to prove the case of $1\leq i\leq n-1$. Fix $1\leq i\leq n-1$. Suppose that $\{\lambda_{1},...,\lambda_{i}\}$ and $\{\lambda_{i+1},\ldots,\lambda_{n+i}\}$ are arbitrary  sequences in $\Lambda$ of length $i$ and $n$, respectively. Since $M$ is a $\mathcal{F}^{n}-$minimal set, we have
\begin{eqnarray*}
f_{\lambda_{n+i}}\circ \cdots\circ f_{\lambda_{i+1}}(f_{\lambda_{i}}\circ \cdots\circ f_{\lambda_{1}}(M))&=& f_{\lambda_{n+i}}\circ \cdots \circ f_{\lambda_{n+1}}(\mathcal{F}_{\sigma_{n}}(M))\\&=& f_{\lambda_{n+i}}\circ \cdots\circ f_{\lambda_{n+1}}(M)\\&\subseteq& \mathcal{F}^{i}(M).
\end{eqnarray*}
This implies that, $\mathcal{F}_{\sigma_{n}}(\mathcal{F}^{i}(M))=\mathcal{F}^{i}(M)$, for all $\sigma\in \Lambda^{\mathbb{Z}_{+}}$.\\
Consider $t\leq n$ as the least positive integer with $\mathcal{F}^{t}(M)\cap M\neq \emptyset$. By contradiction, suppose that $\mathcal{F}_{\sigma_{n}}(A)= A$, for a nonempty set $A\subset \mathcal{F}^{i}(M)$ and a sequence $\sigma\in \Lambda^{\mathbb{Z}_{+}}$. Consider $x\in M$ and the sequence $\{\lambda_{1},\cdots,\lambda_{i}\}$ such that $f_{\lambda_{i}}~\circ \cdots\circ~f_{\lambda_{1}}(x)\in A$. Then  $\mathcal{F}_{\sigma_{n}}(\mathcal{F}^{i}(M))=\mathcal{F}^{i}(M)$ implies that\\ $f_{\lambda_{n+i}}\circ  \cdots \circ f_{\lambda_{i+1}}(\mathcal{F}_{\sigma_{i}}(x))=f_{\lambda_{n+i}}\circ \cdots \circ f_{\lambda_{n+1}}(\mathcal{F}_{\sigma_{n}}(x))\in A$.  \\Since $\mathcal{F}_{\sigma_{n}}(M)=M$, we have $f_{\lambda_{n}}\circ \cdots \circ f_{\lambda_{1}}(x)\in M$  consequently, $\mathcal{F}^{i}(M)\cap M\neq \emptyset$. This is a clear contradiction of the fact that $i\leq t-1$.
\end{proof}
\begin{lemma}\label{lea}
Let $X$ be a compact Hausdorff space and $\mathcal{F}=\{X; f_{\lambda}|\lambda\in\Lambda\} $ be a minimal IFS with continuous functions $f_{\lambda}:X\rightarrow X$. Let $n$ be a positive integer. Then for some $\mathcal{F}^{n}-$minimal set $M$ and some $t\leq n$, we have the following properties:\\
$(1)~X$ is disjoint union of $M, \mathcal{F}(M),\cdots, \mathcal{F}^{t-1}(M).$\\
$(2)$ Each of the sets  $M, \mathcal{F}(M),\cdots, \mathcal{F}^{t-1}(M).$ is clopen.\\
$(3)$ The family $~\{M, \mathcal{F}(M),\cdots, \mathcal{F}^{t-1}(M)\}$ is the collection of all subsets of $X$ that are $\mathcal{F}^{t}-$minimal and also the collection of all subsets of $X$ that are\\ $\mathcal{F}^{n}$-minimal.
\end{lemma}
\begin{proof}
Let $M$ be a $\mathcal{F}^{n}-$minimal set. By Lemma \ref{leaa}, for every positive integer $i$, $\mathcal{F}^{i}(M)$ is a $\mathcal{F}^{n}-$minimal set. Consider $t\leq n$ as the least positive integer with $\mathcal{F}^{t}(M)\cap M\neq \emptyset$. Since minimal sets are disjoint or equal, we have $\mathcal{F}^{t}(M)=M$. By the choice  of $t$, the sets $M, \mathcal{F}(M),\cdots, \mathcal{F}^{t-1}(M)$ are pairwise disjoint. Let $U=M\cup \mathcal{F}(M)\cup\cdots\cup\mathcal{F}^{t-1}(M)$, Our choice  of $t$ and $M$ shows that, $f_{\lambda}(U)\subset U$ for all $\lambda\in \Lambda$. So $U=X$ and the statement $(1)$ holds  .\\
Statements $(2)$ and $(3)$ hold trivially whenever  $X$ is the disjoint union of the sets $M, \mathcal{F}(M),\cdots, \mathcal{F}^{t-1}(M).$
\end{proof}
\begin{corollary}
Let $\mathcal{F}$ be a minimal IFS and $n>1$ be an integer. Then the following are equivalent:\\
$(1)$ There is a continuous map $g:X\rightarrow \mathbb{Z}_{n}$ such that\\ $(g\circ f_{\lambda})(x)=g(x)+1~~(mod~n)$, for all $\lambda\in \Lambda$.\\
$(2)$ There is a proper subset $M$ in $X$ which is  minimal  for $\mathcal{F}^{n}$  and is not minimal for $\mathcal{F}^{t}$, for all $t<n$
\end{corollary}
\begin{proof}
$(1)\Rightarrow (2)$ Consider the map $g:X\rightarrow \mathbb{Z}_{n}$ such that \\$(gof_{\lambda})(x)=g(x)+1~(mod~n)$, for all $\lambda\in \Lambda$ and put $M_{i}=\{x\in X: g(x)=i\}$ for all $0\leq i\leq n-1$. So $M_{0},M_{1},...,M_{n-1}$ are pairwise disjoint nonempty sets. \\Take arbitrary elements $1\leq i\leq n-1$ and $\lambda\in \Lambda$. Let $x\in M_{i}$. Since  \\$(g\circ f_{\lambda})(x)=g(x)+1~~(mod~n)$ we have $f_{\lambda}(x)\in M_{i+1}$. Then  $\mathcal{F}(M_{i})\subset M_{i+1}$. On the other hand, let $y$ be an arbitrary point in $M_{i+1}$. Since $f_{\lambda}$ is a surjective map, there exists $x\in X$ such that $y=f_{\lambda}(x)$. So, $g(y)= g(f_{\lambda}(x))=g(x)+1=i+1$. Hence $x\in M_{i}$. This implies that $M_{i+1}\subseteq f_{\lambda}( M_{i})$. So, $\mathcal{F}(M_{n-1})=M_{0}$ and $\mathcal{F}(M_{i})=M_{i+1}$, for all $0\leq i<n-1$. This implies that $M_{0}$ is a $\mathcal{F}^{n}$-minimal set and is not minimal for $\mathcal{F}^{t}$, where $t<n$.\\
 $(2)\Rightarrow (1)$ Let $M$ be a minimal set for  $\mathcal{F}^{n}$ such that $M$ is not minimal for $\mathcal{F}^{t}$, for all $t<n$. So by statement $(1)$ of  Lemma \ref{lea},  $X$ is a disjoint union of $M, \mathcal{F}(M),..., \mathcal{F}^{n-1}(M).$ Now we define $g:X\rightarrow \mathbb{Z}_{n}$ by $g(x)=i$ if $x\in \mathcal{F}^{i}(M)$, where $\mathcal{F}^{0}(M)=M$. Because of the statement $2$ in Lemma \ref{lea}, $g$ is continuous.
\end{proof}
\begin{definition}
Let $\mathcal{A}=\{X_{1},...,X_{n}\}$ be a finite partition of $X$. We say that the sets $X_{1},...,X_{n}$ are cyclically permuted by $\mathcal{F}$, if $f_{\lambda}(X_{n})=X_{1}$ and \\$f_{\lambda}(X_{i})=X_{i+1}$, for all $1\leq i\leq n-1$ and  all $\lambda\in\Lambda$.
\end{definition}
 It is clear that if the sets $X_{1},...,X_{n}$ are cyclically permuted by $\mathcal{F}$, then $\mathcal{F}_{\sigma_{j}}(X_{i})=X_{i+j~(mod~n)}$ for all $\sigma\in \Lambda^{\mathbb{Z}_{+}}$ and all $i\geq 0$.\\
Next theorem is one of the main results of this paper. This theorem
is an extension of Theorem 2.3 of \cite{[BK]} to iterated function systems and is proved the same way.
\begin{theorem}\label{tha}
 Let $\alpha=(j_{1},j_{2},...)$ be a sequence of integers greater than $1$. Let $m_{i}=j_{1}.j_{2}....j_{i}$ for each $i$. Then, for every $\sigma\in \Lambda^{\mathbb{Z}_{+}}$, $\mathcal{F}_{\sigma}$ is topologically conjugate to $g_{\alpha}$ if and only if the following properties hold:\\
$(1)$ For each positive integer $i$, there is a cover $P_{i}$ of $X$ consisting of $m_{i}$ pairwise disjoint, nonempty, clopen sets which are cyclically permuted by $\mathcal{F}$.\\
$(2)$ For each positive integer $i$, $P_{i+1}$ partitions $P_{i}$.\\
$(3)$ If $V_{1}\supset V_{2}\supset V_{3}\supset ...$ is a nested sequence with $V_{i}\in P_{i}$ for each $i$, then $\bigcap_{i=1}^{\infty}V_{i}$ consists of a single point.
\end{theorem}
\begin{proof}
$(\Rightarrow)$Consider $\sigma\in \Lambda^{\mathbb{Z}_{+}}$ such that $\mathcal{F}_{\sigma}$ is topological conjugate with \\$g_{\alpha}:\Delta_{\alpha}\rightarrow \Delta_{\alpha}$, so there exist $h:X\rightarrow \Delta_{\alpha}$ such that $h \circ \mathcal{F}_{\sigma_{n}}=g_{\alpha}^{n}\circ  h$ for all $n\geq 1$. Consider $Q_{i}=\{Y_{i,0},...,Y_{i,m_{i-1}}\}$ and $P_{i}=\{X_{i,0},...,X_{i,m_{i-1}}\}$, defined in the  proof of Theorem 2.3. in \cite{[BK]}. So, for each $0\leq j\leq m_{i-1}$, $$\mathcal{F}_{\sigma_{j}}(X_{i,0})=h^{-1}\circ g_{\alpha}^{j}\circ h(X_{i,0})=h^{-1}\circ g_{\alpha}^{j}(Y_{i,0})=h^{-1}(Y_{i,j})=X_{i,j}.$$
The other properties are similar to proof of Theorem 2.3. \cite{[BK]}\\
$(\Leftarrow)$  By the proof of Theorem 2.3 in \cite{[BK]}, the set $\cap_{i\geq 1}Y_{i,t_{i}}$ is a singleton, where \\$t_{i}\in \{0,...,m_{i}-1\}$, for all $i\geq 1$.\\
  This property enable us to define $h:X\rightarrow \Delta_{\alpha}$ as follows.\\  Take $x\in X $. For each  $i\geq 1$, there is a unique $k_{i}\leq m_{i}-1$,  such that $x\in X_{i,k_{i}}$. On the other hand, $\cap_{i\geq 1}Y_{i,k_{i}}$ consists a single point $y$. Then, we can put $h(x)=y$. Again, by proof of Theorem 2.3. in \cite{[BK]}, $h$ is continuous and bijective. Now, we show that
$h~\circ\mathcal{F}_{\sigma_{n}}=g_{\alpha}^{n}\circ~h$ for all $n\geq 1$. \\
Let $x\in X$. For each $i\geq 1$, there are unique $X_{i,k_{i}}\in P_{i}$  containing $x$ and  unique $t_{i}\leq m_{i}-1$ congruent to $k_{i}+1$ modulo $m_{i}$. Then, each of the points $h~\circ\mathcal{F}_{\sigma_{1}}(x)$ and $g_{\alpha}\circ~h(x)$ is an element of $\cap_{i\geq 1}Y_{i,t_{i}}$. Since the set $\cap_{i\geq 1}Y_{i,t_{i}}$ consists  a single point, we have $h~\circ\mathcal{F}_{\sigma_{1}}(x)=g_{\alpha}\circ~h(x).$
\\ Repeating the above technique with $f_{\lambda_{1}}(x)$ instead of $x$, we get  $$h~\circ \mathcal{F}_{\sigma_{2}}(x)=h(f_{\lambda_{2}}(f_{\lambda_{1}}(x)))=
g_{\alpha}\circ~h((f_{\lambda_{1}}(x)))=g_{\alpha}(g_{\alpha}(h(x))).$$
So, by induction on $n$ we have
$$h(\mathcal{F}_{\sigma_{n}}(x))=h(f_{\lambda_{n}}(\mathcal{F}_{\sigma_{n-1}}(x)))=
g_{\alpha}(h(\mathcal{F}_{\sigma_{n-1}}(x)))=g_{\alpha}(g_{\alpha}^{n-1}(h(x)))=g_{\alpha}^{n}(h(x)).$$
\end{proof}
The symmetric tent map $T_{a}:[0,1]\rightarrow [0,1]$ with $a\in [0,2]$ is defined as follows:
\[ T_{a}(x) = \left\lbrace
  \begin{array}{c l}
 ax& \text{if ~$~0 \leq x<\frac{1}{2}$},\\
    -ax+a & \text{if ~$~\frac{1}{2}\leq x\leq 1$}
 \end{array}
  \right. \]
\begin{theorem}\cite{[AB]}
Let $\alpha=(j_{1},j_{2},...)$ be a sequence of integers greater that $1$. The set of parameters $a$, such that for the tent map $T_{a}$ the restriction of $T_{a}$ to the closure of the orbit of $c=\frac{1}{2}$ is topologically conjugate to $g_{\alpha}:\Delta_{\alpha}\rightarrow \Delta_{\alpha}$, is dense in $[\sqrt{2},2]$.
\end{theorem}
So, we have the following result.
\begin{corollary}
 Let $\alpha=(j_{1},j_{2},...)$ be a sequence of integers greater than $1$ and $m_{i}=j_{1}.j_{2}....j_{i}$ for each $i$. Let $X$ be a compact Hausdorff space and $\mathcal{F}=\{X; f_{\lambda}|\lambda\in\Lambda\} $ be an IFS with the following properties:\\
$(1)$ For each positive integer $i$, there is a cover $P_{i}$ of $X$ consisting of $m_{i}$ pairwise disjoint, nonempty, clopen sets which are cyclically permuted by $\mathcal{F}$.\\
$(2)$ For each positive integer $i$, $P_{i+1}$ partitions $P_{i}$.\\
$(3)$ If $V_{1}\supset V_{2}\supset V_{3}\supset ...$ is a nested sequence with $V_{i}\in P_{i}$ for each $i$, then $\bigcap_{i=1}^{\infty}V_{i}$ consists of a single point.\\Then the set of parameters $s$ for which the restriction of $T_{s}$ to the closure of the orbit of $c=\frac{1}{2}$ is topologically conjugate to $\mathcal{F}_{\sigma}$, for every $\sigma\in \Lambda^{\mathbb{Z}_{+}}$, is dense in $[\sqrt{2},2]$.
\end{corollary}
Conditions $(1)-(3)$ in Theorem \ref{tha} imply that $X$ is a zero dimensional set. So we have the following result:
\begin{corollary}
Let $X$ be a compact metric space and $\mathcal{F}=\{X; f_{\lambda}|\lambda\in\Lambda\} $ be an \textbf{IFS}. If there exists an adding machine map $g_{\alpha}$ such that for every $\sigma\in \Lambda^{\mathbb{Z}_{+}}$, $\mathcal{F}_{\sigma}$ is topologically conjugate to $g_{\alpha}$ then $X$ is a zero dimensional space.
\end{corollary}
Let $f$ be a continuous map on $X$ and $Q$ be an open cover of $X$. Let $N(Q)$ denote the minimal cardinality of a subcover of  $Q$. The join of open covers $Q_{1}, . . . , Q_{n}$ is $$\bigvee_{k=1}^{n}=\{X_{1},...,X_{n}; X_{k}\in Q_{k}, 1\leq k\leq n\}$$
and put $f^{-1}(Q)=\{f^{-1}(V); V\in Q\}$. Define the topological entropy of $f$ with respect
to an open cover $Q$ by $h(f|Q)=\displaystyle\lim_{n\rightarrow \infty}\frac{1}{n}\log (\bigvee_{k=1}^{n}f^{-k}(Q))$ and the topological entropy of $f$ by $h(f)=\sup_{Q} h(f\mid Q)$ where the supremum is taken over all open covers $Q$ of $X$.
\\
In [ \cite{[HK]}Theorem 11.3.13], the authors prove that if $f$ is a continuous map of an interval with zero topological entropy and $S$ is a closed topological transitive invariant set without periodic orbits, then the restriction of $f$ to $S$ is topologically conjugate to an adding machine map.
\begin{remark}Let $\mathcal{F} $ be an IFS containing only one mapping $f$ which satisfies the conditions (1)-(3) in Theorem \ref{tha} and consider  the sequence $P_{1}\preceq P_{2}\preceq \cdots$ of open covers in these statements. This is clear that   $f^{-k}(P_{i})=P_{i}$ for all $k,i\geq 1$, consequently $h(f|Q)=0$, for all $i\geq 1$. \\Now, let $ Q$ be a finite open cover of $X$. Statement $(3)$ of Theorem \ref{tha} implies that there exists $i_{0}\geq 1$ such that $P_{i_{0}}$ is a refinement of $Q$.
Then \\$h(f|Q)\leq h(f\mid P_{i_{0}})=0$ and consequently   $h(f)=0$. For more details see  Theorem $11.3.13$  of \cite{[HK]}, Corollary $5.4.8$ of \cite{[GA]} and Page $57$ of \cite{[KA]}.\end{remark}
\section{Minimal IFS with regularly recurrent points and adding machine maps}
Suppose that $\mathcal{F}$ is minimal. Denote by $\emph{NM}(\mathcal{F})$ the set of positive integers $i$ such that  some subset $M$ of $X$ is $\mathcal{F}^{i}-$minimal but is not $\mathcal{F}^{i}-$minimal for $j=1,...,i-1$. \\  By Lemma \ref{lea}, for every $n\in \emph{NM}(\mathcal{F})$ there is a unique cover $C_{n}$ of $X$ which consists of the $n-$pairwise disjoint $\mathcal{F}^{n}-$minimal sets.\\
\begin{definition}
The point $x\in X$ is said to be regularly recurrent if for every neighborhood $U$ of $x$, there is a positive integer $n$ such that $\mathcal{F}_{\sigma_{ni}}(x)\in U$, for every $i\geq 0$ and every $ \sigma\in \Lambda^{\mathbb{Z}_{+}}$.
\end{definition}
A point $x\in X$ is called a periodic point of the IFS $\mathcal{F}$ if there is a finite sequence $\lambda_{1},\lambda_{2},...,\lambda_{n}$ of elements in $\Lambda$ such that $f_{\lambda_{n}}\circ f_{\lambda_{n-1}}\circ\cdots f_{\lambda_{1}}(x)=x$ \cite{[B]}.
Our next lemmas are used in the proof of Theorem \ref{ts}.
\begin{lemma}\label{lee}
Suppose that $\mathcal{F}=\{X; f_{\lambda}|\lambda\in\Lambda\} $ is minimal, if $X$ is infinite and $\mathcal{F}$ has some non-periodic regularly recurrent points then $S(\mathcal{F})$ is infinite.
\end{lemma}
\begin{proof}
 Fix $l\geq 1$. We show that there exists $n\in \emph{NM}(\mathcal{F})$ such that $n>l$. \\Let $x$ be a non-periodic regularly recurrent point in $X$. There exists a neighborhood $U\subset X$ of $x$ such that $\mathcal{F}_{\sigma_{i}}(x)\notin\overline{U}$, for each $\sigma\in \Lambda^{\mathbb{Z}_{+}}$ and $i=1,...,l$. Since $x$ is a regularly recurrent point, there is a positive integer $t$ such that for each non-negative integer $i$, and each $\sigma\in \Lambda^{\mathbb{Z}_{+}}$ we have $\mathcal{F}_{\sigma_{ti}}(x)\in U$.\\So Lemma \ref{lea} implies that the set $M= \overline{\{\mathcal{F}_{\sigma_{ti}}(x)\}_{\sigma\in \Lambda^{\mathbb{Z}_{+}}, i>1}}$ is $\mathcal{F}^{t}-$minimal. Consider $n$ as the least positive integer for which $M$ is $\mathcal{F}^{n}-$minimal. So, $n\in S$ and $M\subset \overline{U},$ hence $n>l$.
\end{proof}
\begin{lemma}\label{lef}
Suppose that $\mathcal{F}=\{X; f_{\lambda}|\lambda\in\Lambda\} $ is minimal, $X$ is infinite and $\mathcal{F}$ has some non-periodic regularly recurrent points. Suppose that $l,~n\in S$ and $l$ is a multiple of $n$, then $C_{l}$ refines $C_{n}$.
\end{lemma}
\begin{proof}
Let $M\in C_{l}$ and $x\in M$. Since $\overline{\{\mathcal{F}_{\sigma_{li}}(x)\}}_{\sigma\in \Lambda^{\mathbb{Z}_{+}},i\geq 0}\subset M$ is a $\mathcal{F}^{l}-$minimal set. So $M=\overline{\{\mathcal{F}_{\sigma_{li}}(x)\}}_{\sigma\in \Lambda^{\mathbb{Z}_{+}},i\geq 0}$. This is clear that \\ $\{\mathcal{F}_{\sigma_{li}}(x)\}_{\sigma\in \Lambda^{\mathbb{Z}_{+}},i\geq 0}\subset\{\mathcal{F}_{\sigma_{ni}}(x)\}_{\sigma\in \Lambda^{\mathbb{Z}_{+}},i\geq 0}$. So $M$ is contained in $\overline{\{\mathcal{F}_{\sigma_{ni}}(x)\}}_{\sigma\in \Lambda^{\mathbb{Z}_{+}},i\geq 0}$ which is a $\mathcal{F}^{l}-$minimal set and consequently an element of $C_{n}$.
\end{proof}
\begin{lemma}\label{leg}
Suppose that $\mathcal{F}=\{X; f_{\lambda}|\lambda\in\Lambda\} $ is minimal, $X$ is infinite and $\mathcal{F}$ has some non-periodic regularly recurrent points. Let $l\in S$ and $n$ be a positive integer. If $l$ is a multiple of $n$, then $n\in \emph{NM}(\mathcal{F})$.
\end{lemma}
\begin{proof}
Let $l=kn$ and $M$ be a $\mathcal{F}^{l}-$minimal set. Put \\$M_{1}=M\cup \mathcal{F}^{n}(M)\cup...\cup \mathcal{F}^{(k-1)n}(M).$ So, for all $1\leq i\leq n-1$,\\ $\mathcal{F}^{i}(M_{1})\cap M_{1}=\emptyset$ and these facts that $\mathcal{F}^{l}(M)\subseteq M$ and $l=kn$ imply that $\mathcal{F}^{n}(M_{1})\subseteq \bigcup_{j=1}^{k}\mathcal{F}^{jn}(M)\subseteq M_{1}$, which completes the proof.
\end{proof}
\begin{lemma}\label{leh}
Suppose that $\mathcal{F}=\{X; f_{\lambda}|\lambda\in\Lambda\} $ is minimal, $X$ is infinite and $\mathcal{F}$ has some non-periodic regularly recurrent points. Let $l$ and $n$ be two prime numbers in $\emph{NM}(\mathcal{F})$. Then $ln\in \emph{NM}(\mathcal{F})$.
\end{lemma}
\begin{proof}
Let $M_{1}$ be a $\mathcal{F}^{l}-$minimal set and $x\in M_{1}$ be arbitrary. There exists a $\mathcal{F}^{n}-$minimal set $M_{2}$ such that, $x\in M_{2}$. Consider $A=M_{1}\cap M_{2}$. If \\$\mathcal{F}^{t}(A)\cap A\neq\emptyset$, for some $t>0$, then $\mathcal{F}^{t}(M_{2})\cap M_{2}\neq\emptyset$. This implies that $t$ is a multiple of $n$. By similar argument for $M_{1}$ we have $\mathcal{F}^{ln}(A)\cap A\neq\emptyset$ and $\mathcal{F}^{i}(A)\cap A=\emptyset$ for all $1\leq i\leq ln-1$.
\end{proof}
In [\cite{[BK]} Theorem 2.4.], the authors prove that:\\
If $f:X\rightarrow X$ is minimal, $X$ is infinite and $f$ has some non-periodic regularly recurrent points, then there is a sequence $\alpha$ of prime numbers and a continuous surjective map $\pi:X\rightarrow  \Delta_{\alpha}$ such that $g_{\alpha}\circ\pi=\pi \circ f$, for all $\lambda\in \Lambda$. \\
In the next theorem, we have a similar result for iterated function systems.
\begin{theorem}\label{ts}
 Suppose that $\mathcal{F}=\{X; f_{\lambda}|\lambda\in\Lambda\} $ is minimal, $X$ is infinite and $\mathcal{F}$ has some non-periodic regularly recurrent points. Then there is a sequence $\alpha$ of prime numbers and a continuous surjective map $\pi:X\rightarrow  \Delta_{\alpha}$ such that $g_{\alpha}\circ\pi=\pi \circ f_{\lambda}$, for all $\lambda\in \Lambda$. Also, $\pi^{-1}(\pi(x))=\{x\}$, for every regularly recurrent point $x\in X$.
\end{theorem}
\begin{proof}
Let  $p$ be a prime number  and $s\in \emph{NM}(\mathcal{F})$. Consider $N(s,p)$ as the multiplicity of $p$ in the prime factorization of $s$ and $N(p)=\max_{s\in S}N(s,p).$ Now assume that $\alpha=(p_{1},p_{2},...)$ is a sequence of primes such that the number of appearance of each prime number $p$ appears in this sequence is exactly $N(p)$. Therefore if we put $n_{i}=p_{1}p_{2}...p_{i}$  then Lemmas \ref{leg} and \ref{leh} imply that \\$n_{i}\in \emph{NM}(\mathcal{F})$, for every positive integers $i$. Consider $\{X_{i,1},...,X_{i,n_{i}}\}$ as $C_{n_{i}}$ cover of $X$. By Lemma \ref{lea}, the statement $(1)$ of Theorem \ref{tha} holds and by Lemma \ref{lea}, the statement $(2)$ of Theorem \ref{tha} holds. So by Theorem \ref{tha} for every $\sigma\in \Lambda^{\mathbb{Z}_{+}}$, there exists a continuous surjective map $\pi:X\rightarrow \Delta_{\alpha}$ such that $g_{\alpha}^{n}\circ \pi=\pi \circ \mathcal{F}_{\sigma_{n}}$ for all $n\geq 1$. Now we show that for every regularly recurrent point $x\in X$, $\pi^{-1}(\pi(x))=\{x\}$.\\
Let $y\in \pi^{-1}(\pi(x))$ and $y\neq x$. So there exists a sequence $V_{1}
\supset V_{2}\supset V_{3}\supset ....$ with $V_{i}\in \{X_{i,1},X_{i,2}, ..., X_{i,m_{i}}\}$ for each $i$, such that $x, y\in \cap _{i\geq 1} V_{i}$.\\ Let  $V$ be an open subset of $X$ such that $x\in V$ and  $y \notin \overline{V}$. Since $x$ is regularly recurrent there is a positive integer $n$ such that
 $\mathcal{F}_{\sigma_{ni}}(x)\in V$ for all $i\geq 1$. So $M=\{\mathcal{F}_{\sigma_{ni}}(x)\}_{\sigma\in \Lambda^{\mathbb{Z}_{+}},i\geq 1}$
is an $\mathcal{F}^{n}-$ minimal set in $\overline{V}$. This immediately implies that there exists  $1\leq t\leq n$ such that $t\in \emph{NM}(\mathcal{F})$
 and $M$ is $\mathcal{F}^{t}-$ minimal set.
On the other hand, there is a positive integer
$j$ such that $n_{j}$ is a multiple of $t$.
So, by Lemma \ref{lef} and this fact that $x\in V_{j}\cap M$ we  have $V_{j}\subset M$,  which contradicts our assumption that $y\in V_{j}$ and $y \notin M$. Thus $\pi^{-1}(\pi(x))=\{x\}$.
\end{proof}
\begin{corollary}
The IFS $\mathcal{F}$ is minimal and each point of $X$ is regularly recurrent if and only if there is a sequence $\alpha$ of prime numbers such that $\mathcal{F}_{\sigma}$ is topologically conjugate to $g_{\alpha}$, for every $\sigma\in  \Lambda^{\mathbb{Z}_{+}}$.
\end{corollary}
Let $g:X\longrightarrow X$ be a continuous map, the $\omega$-limit set of $g$ at $y$ is
$$\omega(y, g)=\bigcap_{m \geq0}(\overline{\bigcup_{k\geq m} g^{k}(y)}).$$
\begin{definition}\cite{[DHSA]}
 Let X be a compact metric space, with $\alpha = (j_{1},j_{2}, . . .)$ a sequence of integers so that $j_{i} \geq 2$ for each $i$.
Denote by
$S_{\alpha}(X) $ the set of all pairs $(y, g)$ in $X \times C(X ,X)$ such that $ (\omega(y, g), g)$ is topologically conjugate to $g_{\alpha}:\Delta_{\alpha}\longrightarrow \Delta_{\alpha}$.
\end{definition}
\begin{corollary}\cite{[BK]}
Let $\beta = (j_{1}, j_{2}, . . .)$ and $\gamma = (k_{1}, k_{2}, . . .)$ be sequences of integers with
$j_{i} \geq 2$ and $k_{i} \geq 2$ for each $i$. We let $M_{\beta}$ denote a function whose domain is the set of
all prime numbers and which maps to the extended natural numbers ${0, 1, 2, . . .,\infty}$. The
function $M_{\beta}$ is defined by $M_{\beta}(p) =\Sigma_{i=1}^{\infty}n_{i}$
where $n_{i}$ is the power of the prime $p$ in the prime factorization of $j_{i}$.
Then $f_{\beta}$ and $f_{\gamma}$ are topologically conjugate if and only if the functions $M_{\beta}$ and $M_{\gamma}$ are
equal.
\end{corollary}
Then we have the following corollary.
\begin{corollary}
Suppose that the IFS $\mathcal{F}$ is minimal and each point of $X$ is regularly recurrent. There is a sequence $\alpha$ of prime numbers such that for every sequence $\beta$ of prime numbers,  $M_{\beta}=M_{\alpha}$ implies that $\mathcal{F}_{\sigma}$ is topologically conjugate to $g_{\beta}$, for every $\sigma\in  \Lambda^{\mathbb{Z}_{+}}$.
\end{corollary}
Let $\alpha = (j_{1},j_{2}, . . .)$ a sequence of integers so that $j_{i} \geq 2$ for each $i$. In \cite{[DHS]}, the authors prove that  if for every prime number $p$, $M_{\alpha}(p)<\infty$  then  $S_{\alpha}(X)$ is a nowhere dense subset of  $ X �\times C(X,X)$.\\
So, by considering Lemmas \ref{lef} and \ref{leg} and the sequence $\alpha$ in the proof of Theorem \ref{ts} we have the following result.
\begin{corollary}
Suppose that the IFS $\mathcal{F}$ is minimal and each point of $X$ is regularly recurrent. There exists  a sequence $\alpha$ of prime numbers such that the set of all
$(y, g) \in X �\times C(X ,X)$ which $g:\omega(y, g)\longrightarrow \omega(y, g)$, for every $\sigma\in  \Lambda^{\mathbb{Z}_{+}},$ is topologically conjugate to $\mathcal{F}_{\sigma} $ is a nowhere dense subset of  $ M \times C(M,M).$
\end{corollary}
In [\cite{[SO]}, Corollary 4.3] the authors prove that if $f:X\rightarrow X$ is a sensitive non-wandering dynamical system with the shadowing property then the set of non-periodic regularly points of $f$ is  dense in  $X$.  So we have the following result.

\begin{theorem}\label{tht}
Suppose that $f:X\rightarrow X$ is sensitive and minimal map with the shadowing property. Then there is a sequence $\alpha$ of prime numbers and a continuous map $\pi:X\rightarrow  \Delta_{\alpha}$ such that $g_{\alpha}\circ\pi=\pi \circ f$ and $\pi$ is one to one on a dense subset of $X$.
\end{theorem}
\begin{proof}
Minimality of $f$ implies that $\overline{\{f^{n}(x)\}_{n\geq 0}}=X$ for all $x\in X$.  So $f$ is a non-wandering system and does not have any periodic point. Then by Corollary 4.3. of \cite{[SO]} the set of all regularly recurrent points is dense in $X$. Consequently Theorem \ref{ts} (in the case of $\Lambda$ has only one element), implies that there is a sequence $\alpha$ of prime numbers and a continuous map $\pi:X\rightarrow  \Delta_{\alpha}$ such that $g_{\alpha}\circ\pi=\pi \circ f$ and $\pi$ is one to one on the set of all regularly recurrent points which a dense subset of $X$.
\end{proof}
 So it seems plausible
to put forward the following conjectures.\\
\textbf{Conjecture $1$:}
 Suppose that $f:X\rightarrow X$ is a sensitive  minimal map with the shadowing property. Then there is a sequence $\alpha$ of prime numbers and a homeomorphism map $\pi:X\rightarrow  \Delta_{\alpha}$ such that $g_{\alpha}\circ\pi=\pi \circ f$.
\\
\textbf{Conjecture $2$}:
 Suppose that $\mathcal{F}=\{X; f_{\lambda}|\lambda\in\Lambda\} $   is a sensitive minimal IFS map with the shadowing property. Then there is a sequence $\alpha$ of prime numbers and a homeomorphism map $\pi:X\rightarrow  \Delta_{\alpha}$ such that $g_{\alpha}\circ\pi=\pi \circ f_{\lambda}$, for all $\lambda\in \Lambda$.


\begin{thebibliography}{99}
\bibitem{[AB]} {\sc L. Alvin and K. Brucks},
 {\em   Adding machines, Kneading maps and Endpoints}. Topology and its Applications. \textbf{158}, (2011) 542-550.
\bibitem{[B]} {\sc M.F. Barnsley},
 {\em,  Fractals everywhere}, Academic Press, Boston 1988.
\bibitem{[BV]}  {\sc M.F. Barnsley  and A. Vince},
 {\em  The Conley attractor of an iterated function system}, Bull. Aust. Math. Soc.   \textbf{88}, (2013)  267-279.
\bibitem{[BIR]} {\sc G.D. Birkhoff},
 {\em Dynamical Systems}, Amer. Math. Soc. Colloq. Publ., vol. \textbf{9}, American Mathematical Society, Providence, RI, 1966.
\bibitem{[BK]} {\sc L. Block and J. Keesling},
 {\em A characterization of adding machine maps}, Topology and its Applications. \textbf{140}, (2004) 151-161.
\bibitem{[DHSA]}{\sc  E. D'Aniello, Paul D. Humke and T.H. Steele},
 {\em Ubiquity of odometers in topological dynamical systems}, Topology and its Applications. \textbf{156}, (2008) 240-245.
\bibitem{[DHS]} {\sc E. D'Aniello, Paul D. Humke and T.H. Steele},
 {\em The space of adding machines generated by continuous self maps of manifolds}, Topology and its Applications. \textbf{157}, (2010) 954-960.
\bibitem{[FN]} {\sc M. Fatehi Nia,  Iterated Function Systems with the Shadowing Property},  J. Adv. Res. Pure Math.   \textbf{7}, (2015)  83-91.
\bibitem{[GHS]}{\sc F.H. Ghane, A.J. Homburg, A. Sarizadeh},
 {\em C1-robustly minimal iterated function systems}, Stoch.
Dyn. 10 (2010), n0. 1, 155-160.
\bibitem{[GA]}{\sc M. Gidea and C.P. Niculescu,} {\em Chaotic Dynamical Systems. An Introduction}  Craiova; Universitaria Press, 2002.
\bibitem{[GG]} {\sc V. Glavan and V. Gutu},
 {\em Shadowing in parameterized IFS, Fixed Point Theory}, \textbf{7},  (2006)  263-274.
\bibitem{[GG1]}{------},
 {\em Attractors and fixed points of weakly contracting relations}, Fixed Point Theory,  \textbf{5},  (2004)  265-284.
\bibitem{[HK]} {\sc B. Hasselbatt and A. Katok},
 {\em A first course in dynamics, with a panorama of recent developments}, Cambridge University Press. 2003.
\bibitem{[KA]} {\sc  A. Katok},
 {\em Ergodic theory and dynamical systems II }, Proceedings Special Year Maryland 1979-80.
\bibitem{[SO]} {\sc  T.K. Subrahmonian Moothathu and P. Oprocha},
 {\em Shadowing, entropy and minimal subsystems }, Monatsh Math \textbf{172}, (2013) 357-378.
\end{thebibliography}
\end{document}